\documentclass[12pt]{article}
\usepackage{amsmath,amssymb,amsthm,stmaryrd}
\usepackage[all]{xy}

\newtheorem{thm}{Theorem}[section]
\newtheorem{cor}[thm]{Corollary}
\newtheorem{prop}[thm]{Proposition}

\theoremstyle{definition}
\newtheorem{defn}[thm]{Definition}
\theoremstyle{remark}
\newtheorem{rmk}[thm]{Remark}
\newtheorem{exam}[thm]{Example}

\def\co{\colon\thinspace}
\newcommand{\mb}[1]{\mathbb{#1}}

\newcommand{\Hom}{\ensuremath{{\rm Hom}}}

\newcommand{\overto}{\mathop\rightarrow}

\newcommand{\GL}{{\rm GL}}

\newcommand{\Spec}{{\rm Spec}}
\newcommand{\Spf}{{\rm Spf}}

\newcommand{\Wt}{{\mb W}}

\newcommand{\comp}[1]{\ensuremath{#1^\wedge}}

\newcommand{\pow}[1]{\left\llbracket{#1}\right\rrbracket}

\title{Structured ring spectra and displays}
\author{Tyler Lawson}

\begin{document}
\maketitle

\begin{abstract}
We combine Lurie's generalization of the Hopkins-Miller theorem with work of Zink-Lau on displays to give a functorial construction of even-periodic $E_\infty$ ring spectra $E$, concentrated in chromatic layers $2$ and above, associated to certain $n \times n$ invertible matrices with coefficients in the Witt ring of $\pi_0(E)$.  This is applied to examples related to Lubin-Tate and Johnson-Wilson spectra.  We also give a Hopf algebroid presentation of the moduli of $p$-divisible groups of height greater than or equal to $2$.
\end{abstract}

\section{Introduction}

One of the most successful methods for understanding stable homotopy theory is its connection to formal groups.  By work of Quillen, a homotopy commutative and associative ring spectrum $R$ has $MU$-homology $MU_* R$ an algebra over the the Lazard ring $L$, which carries a universal $1$-dimensional formal group law, and the $MU$-homology cooperations precisely provide $MU_* R$ with rules for change-of-coordinates on the formal group law.

In recent years much study has been devoted to the study of the converse problem.  Given a ring $R$ with a formal group on $R$, can we reconstruct a ring spectrum $E$ whose associated formal group lifts that on $R$?  In addition, can more rigid structure (such as the structure of an $E_\infty$ algebra) be imposed on $E$?  Can these constructions be made functorial?

In 2005 Lurie announced a theorem that lifts formal groups to $E_\infty$ ring structures, generalizing the Goerss-Hopkins-Miller theorem \cite{goersshopkins}.  The application of this theorem requires extra data: an extension of the formal group to a $p$-divisible group.  In addition, the $p$-divisible group on $R$ is required to satisfy a universality condition at each point of $\Spec(R)$.  This specifically can be applied to produce  elliptic cohomology theories and the theory of topological modular forms, and served as the basis for previous joint work with Behrens generalizing topological modular forms to moduli of higher-dimensional abelian varieties that reach higher chromatic levels \cite{taf}.

From a geometric point of view these are some of the most natural families of $1$-dimensional $p$-divisible groups, as the study of $p$-divisible groups originated with their connection to abelian varieties.  However, one of the major obstructions to understanding the associated cohomology theories is that such an understanding rests on an understanding of certain moduli of higher-dimensional abelian varieties, and in particular their global geometry.  This presents barriers both because of the necessary background and because these moduli seem to be intrinsically difficult.  Moreover, from the point of view of homotopy theory one does not have as many of the ``designer'' tools of the subject \cite{mikeRW}, which construct spectra with the explicit goal of capturing certain homotopy-theoretic phenomena.

The aim of this paper is to provide a method for constructing $E_\infty$ ring spectra from purely algebraic data.  Specifically, theorem~\ref{thm:main} allows the functorial construction of even-periodic $E_\infty$ ring spectra with $\pi_0 = R$ from the data of certain $n \times n$ invertible matrices with coefficients in the Witt ring of $R$.  This is obtained using Zink's {\em displays} on $R$ \cite{zink}, which correspond to certain $p$-divisible groups over $\Spec(R)$.  This generalizes the Dieudonn\'e correspondence between $p$-divisible groups over a perfect field and their associated Dieudonn\'e modules, and in particular restricts to this structure at any residue field $k$.  Due to restrictions on the $p$-divisible groups constructible by these displays, the associated spectra are concentrated in chromatic height greater than or equal to $2$.

The layout of this paper is as follows.  In section~\ref{sec:defs} we recall the definitions of displays and nilpotent displays over a ring $R$ from \cite{zink}, specifically concentrating on those in matrix form.  We state the equivalence of categories between nilpotent displays and formal $p$-divisible groups on $\Spec(R)$ due to Zink and Lau.  In section~\ref{sec:moduli} we apply Serre duality to obtain a classification of $p$-divisible groups of dimension $1$ and formal height $\geq 2$ on $\Spec(R)$, and give a presentation of the moduli of such $p$-divisible groups by a large Hopf algebroid.  In section~\ref{sec:deform} we study the deformation theory of nilpotent displays in matrix form over a ring $R$ and use this to give a criterion for these to satisfy the universal deformation criterion.  Specifically, a display in matrix form determines a map from $\Spec(R)$ to projective space $\mb P^{h-1}$ that is \'etale if and only if the associated $p$-divisible group is locally a universal deformation.  In section~\ref{sec:spectra} we recall the statement of Lurie's theorem and apply it to functorially obtain even-periodic ring spectra associated to certain displays.  This is applied to construct ``almost-global'' objects which are related to Lubin-Tate and Johnson-Wilson spectra.  Finally, in section~\ref{sec:period} we relate this to the work of Gross and Hopkins on the rigid analytic period map \cite{grosshopkins}.  In the specific case of a Lubin-Tate formal group over $R = \Wt(k)\pow{u_i}$, there is a choice of coordinates such that the above map $\Spec(R) \to \mb P^{h-1}$ induces a rigid-analytic map that agrees with the Gross-Hopkins period map modulo the ideal $(u_i)^p$.

We mention that Zink's theory of Dieudonn\'e displays provides objects equivalent to general $p$-divisible groups over certain complete local rings.  These could be applied to produce spectra associated to universal deformations of $p$-divisible groups of dimension 1 over a field, analogous to Lubin-Tate spectra, which are worth study in their own right.  However, our goal in this paper is to allow more global rather than local constructions.

{\bf Notation}.  For a ring $R$, we write $\Wt(R)$ for the ring of $p$-typical Witt vectors over $R$.  This carries Frobenius and Verschiebung maps $f, v\co \Wt(R) \to \Wt(R)$, the Teichm\"uller lift $[-]\co R \to \Wt(R)$, and ghost maps $w_k\co \Wt(R) \to R$.  We write $I_R$ for the ideal of definition $v (\Wt(R))$.

\section{Displays}
\label{sec:defs}

We will first briefly recall the classical Dieudonn\'e correspondence.  Let $k$ be a perfect field. The Dieudonn\'e module functor $\mb D$ is a contravariant equivalence of categories between $p$-divisible groups over $\Spec(k)$ and finitely generated free modules $M$ over the Witt ring $\Wt(k)$ that are equipped with operators $F$ and $V$ satisfying the following properties.
\begin{itemize}
\item $F$ is semilinear: for $x \in \Wt(k)$, $m \in M$, $F(xm) = f(x) F(m)$.
\item $V$ is anti-semilinear: for $x \in \Wt(k)$, $m \in M$, $x V(m) = V(f(x) m)$.
\item $FV = VF = p$.
\end{itemize}

The inverse equivalence can be made explicit.  The group schemes $\Spec(W_n)$ representing Witt vectors assemble into an inductive system using the Verschiebung maps, and the colimit (the ``unipotent'' Witt covectors) is a formal group scheme with self-maps $F$ and $V$ and an action of the Witt vectors $\Wt(k)$.  For $p$-divisible groups which accept no nontrivial maps from the multiplicative group scheme $\mb G_m$, $\mb D(\mb G)$ is the set of maps from $\mb G$ to this direct limit.  The generators of the Dieudonn\'e module provide an embedding of $\mb G$ into a product of copies of the Witt covectors.  The theory for general $\mb G$ is harder, and the unipotent Witt covectors must be replaced by a suitable completion viewed as a sheaf on commutative $k$-algebras \cite{fontaine}. 

Classical Dieudonn\'e theory also incorporates duality.  Each $p$-divisible group $\mb G$ has a Serre dual $\mb G^\vee$, whose associated Dieudonn\'e module $\mb D(\mb G^\vee)$ is isomorphic to the dual module
\[
\mb D(\mb G)^t = \Hom_{\Wt(k)}(\mb D(\mb G), \Wt(k))
\]
equipped with Frobenius operator $V^t$ and Verschiebung operator $F^t$.  The {\em covariant} Dieudonn\'e module of $\mb G$ is the Dieudonn\'e module of $\mb G^\vee$, and this provides a covariant equivalence of categories between Dieudonn\'e modules and $p$-divisible groups.

Zink's theory of displays is a generalization of the Dieudonn\'e correspondence.  However, over a non-perfect ring $R$ defining the map $V$ is no longer sufficient.  This is instead replaced by a choice of ``image'' of $V$, together with an inverse function $V^{-1}$.

\begin{defn}[\cite{zink}, Definition 1]
A {\em display} over a ring $R$ consists of a tuple $(P,Q,F,V^{-1})$, where $P$ is a finitely generated locally free $\Wt(R)$-module, $Q$ is a submodule of $P$, and $F\co P \to P$ and $V^{-1}\co Q \to P$ are  Frobenius-semilinear maps.  These are required to satisfy the following:
\begin{itemize}
\item $I_R P \subset Q \subset P$,
\item the map $P/I_RP \to P/Q$ splits,
\item $P$ is generated as a $\Wt(R)$-module by the image of $V^{-1}$, and
\item $V^{-1}(v(x) y) = x F(y)$ for all $x \in \Wt(R)$, $y \in P$.
\end{itemize}
\end{defn}

If $p$ is nilpotent in $R$, then $\Wt(R) \to R$ is a nilpotent thickening, and $P$ being locally free on $\Wt(R)$ implies this is also true locally on $\Spec(R)$.  In addition, $P/Q$ and $Q/I_R P$ are locally free $R$-modules.

Therefore, locally on such $R$ we may choose a basis $\{e_i | 1 \leq i \leq h\}$ of $P$ such that $Q = I_R P + \langle e_{d+1}, \cdots,  e_h\rangle$.  We refer to $h$ as the height and $d$ as the dimension of the display, and these are locally constant on $\Spec(R)$.  As in \cite{zink}, in such a basis we may define an $h \times h$ matrix $(b_{ij})$ as follows:
\begin{eqnarray*}
F e_j &=& \sum b_{ij} e_i \text{ for }1\leq j \leq d,\\
V^{-1} e_j &=& \sum b_{ij} e_i \text{ for }(d+1)\leq j \leq h.
\end{eqnarray*}
These determine all values of $F$ and $V^{-1}$:
\begin{eqnarray*}
  F e_j = V^{-1}(v(1) \cdot e_j) &=& \sum (p b_{ij}) e_i\text{ for }(d+1) \leq j \leq h,\\
  V^{-1}(v(x) \cdot e_j) &=& \sum (x b_{ij}) e_i\text{ for }1 \leq j \leq d, x \in \Wt(R).
\end{eqnarray*}
(We remark that the reason this last statement includes an indeterminate $x$ is that $Q$ is often not free as a $\Wt(R)$-module.)

The image of $V^{-1}$ generates all of $P$ if and only if this matrix is invertible, or equivalently that its image under the projection $M_h \Wt(R) \to M_h R$ is invertible.

In block form, we may write $(b_{ij})$ as
\[
\left[\begin{array}{c|c}
  u_1 & u_2
\end{array}\right],
\]
and find that in this basis the functions $F$ and $V^{-1}$ are given by block multiplication as follows:
\begin{eqnarray*}
F
\left[\begin{array}{c}
  x \\\hline y
\end{array}\right] = 
\left[\begin{array}{c|c}
  u_1 & p u_2
\end{array}\right]
\left[\begin{array}{c}
  fx \\\hline fy
\end{array}\right],
\\
V^{-1}
\left[\begin{array}{c}
  v x \\\hline y
\end{array}\right] = 
\left[\begin{array}{c|c}
  u_1 & u_2
\end{array}\right]
\left[\begin{array}{c}
  x \\\hline fy
\end{array}\right].
\end{eqnarray*}

To aid calculation in section~\ref{sec:deform} and further, we will refer to the {\em inverse} matrix $B = (b_{ij})^{-1}$ as a matrix form for the given display.  Specifying the matrix form is equivalent to specifying the inverse matrix.

If $R$ is a perfect field of characteristic $p$, the operator $V^{-1}$ has a genuine inverse defining an anti-semilinear map $V\co P \to Q \subset P$, and the maps satisfy $VF = FV = p$. Under these circumstances, if
\[
B^{-1} = \left[\begin{array}{c|c}
  u_1 & u_2
\end{array}\right],
\ B = \left[\begin{array}{c}
  w_1 \\ \hline w_2
\end{array}\right]
\]
are block forms, then the operators reduce to the classical operators $F$ and $V$ on a Dieudonn\'e module which is free over $\Wt(k)$ with basis $\{e_i\}$, and these operators have the matrix expression
\[
F 
\left[\begin{array}{c}
  x \\\hline y
\end{array}\right]
= 
\left[\begin{array}{c|c}
  u_1 & p u_2
\end{array}\right]
\left[\begin{array}{c}
  fx \\\hline fy
\end{array}\right],
\ V x
 = \left[\begin{array}{c}
v w_1 \\ \hline f^{-1} w_2
\end{array}\right]
(f^{-1} x).
\]

A map between two displays is a $\Wt(R)$-module map $P \to P'$ preserving submodules and commuting with $F$ and $V^{-1}$.  If $P$ and $P'$ have bases $\{e_i\}$ and $\{e_i'\}$ as above, a map $\phi\co P \to P'$ takes $Q$ to $Q'$ if and only if, when we write $\phi(e_j) = \sum \phi_{ij} e_i'$, we have $\phi_{ij} \equiv 0$ in $\Wt(R)/I_R$ when $i \geq (d+1)$, $j \leq d$.  Given an isomorphism $\phi$ from $(P,Q,F, V^{-1})$ to $(P', Q', F', (V')^{-1})$, the operators $F'$ and $(V')^{-1}$ are determined uniquely by $F' = \phi F \phi^{-1}$ and similarly for $V'$.  If these displays have matrix forms $B$ and $B'$ respectively, and we write $\phi$ in the block form
\[
\phi = 
\left[\begin{array}{c|c}
  a & vb \\
\hline
  c & d
\end{array}\right],
\]
then we find by direct calculation that the matrix form for the display $(P',Q',\phi F \phi^{-1}, \phi V^{-1} \phi^{-1})$ is given by the change-of-coordinates formula
\begin{equation}
  \label{eq:changeofbasis}
B' = 
\left[\begin{array}{c|c}
  fa & b \\
\hline
  p \cdot fc & fd
\end{array}\right]
\cdot B \cdot
\left[\begin{array}{c|c}
  a & vb \\
\hline
  c & d
\end{array}\right]^{-1}.
\end{equation}
The associated map of displays induces the map of modules $Q/I_RP \to Q'/I_RP'$ given in matrix form by $w_0(d)$.

Given a matrix form $B$, let $\overline B$ be the $(h-d) \times (h-d)$ matrix in $R/(p)$ which is the image of lower-right corner of $B$ under the projection $\Wt(R) \to R/(p)$.  We say that the display is {\em nilpotent} if the product
\[
f^n \overline B \cdots f^2 \overline B \cdot f \overline B \cdot \overline B
\]
is zero for some $n \geq 0$.  (This is independent of the choice of basis, as it is equivalent to the semilinear Frobenius map acting nilpotently on the quotient of $Q$ by $(p) + I_R P$.)  Here $f$ is the Frobenius on $R/(p)$ applied to each entry of the matrix.  A general display over $R$ is nilpotent if it is locally nilpotent in the Zariski topology.  We refer to a display on a formal $\mb Z_p$-algebra $R = \lim R_i$ as nilpotent if its restrictions to the $R_i$ are nilpotent.

\begin{thm}[\cite{zink}, \cite{lau}]
\label{thm:zinkcovariant}
If $R$ is a formal $\mb Z_p$-algebra, there is a (covariant) equivalence of categories between nilpotent displays over $R$ and formal $p$-divisible groups on $\Spec(R)$.

Under this correspondence, the Lie algebra of the $p$-divisible group associated to a display $(P,Q,F,V^{-1})$ is the locally free $R$-module $P/Q$.
\end{thm}

\section{Moduli of $1$-dimensional $p$-divisible groups}
\label{sec:moduli}

Beginning in this section we specialize to the case of topological interest: the theory of $1$-dimensional $p$-divisible groups.  Unfortunately, there are very few $1$-dimensional $p$-divisible groups to which Theorem~\ref{thm:zinkcovariant} applies.  The only ones satisfying the conditions of Lurie's theorem (see \ref{thm:lurie}) are analogues of the Lubin-Tate formal groups.

However, the category of $p$-divisible groups has a notion of duality, compatible with a corresponding duality on the display.  Serre duality is a contravariant self-equivalence of the category of $p$-divisible groups over a general base $X$ that associates to a $p$-divisible group $\mb G$ of constant height $h$ and dimension $d$ a new $p$-divisible group $\mb G^\vee = \Hom(\mb G, \mb G_m)$ of height $h$ and dimension $h-d$.  This equivalence takes formal $p$-divisible groups to $p$-divisible groups with no subobjects of height $1$ and dimension $1$.  There is a compatible notion of duality for displays \cite[1.13, 1.14]{zink}, sending a display $(P,Q,F,V^{-1})$ to a new display $(P^t, Q^t, F^t, V^{-t})$ where $P^t = \Hom(P,\Wt(R))$ and $Q^t$ is the submodule of maps sending $Q$ into $I_R$.  The operators $F^t$ and $V^{-t}$ are determined by the formula $v((V^{-t} f)(V^{-1} x)) = f(x)$ for $f \in Q^t$, $x \in Q$.

Composing this duality equivalence with Theorem~\ref{thm:zinkcovariant}, we find the following.

\begin{cor}
\label{thm:zinkcontravariant}
If $R$ is a formal $\mb Z_p$-algebra, there is a contravariant equivalence of categories between nilpotent displays of height $h$ and dimension $(h-1)$ over $R$ and $p$-divisible groups of dimension $1$ and formal height $\geq 2$ on $\Spec(R)$.

Under this correspondence, the Lie algebra of the $p$-divisible group associated to a display $(P,Q,F,V^{-1})$ is the locally free $R$-module $\Hom_R(Q/I_RP,R)$, and the module of invariant $1$-forms is isomorphic to $Q/I_RP$.
\end{cor}

\begin{rmk}
\label{rmk:oneform}
In particular, the $p$-divisible group associated to a display in matrix form $B$, with basis $e_1 \ldots e_h$, has a canonical nowhere-vanishing invariant $1$-form $u$ which is the image of $e_h$ in $Q/I_RP$, and a change-of-coordinates as in equation (\ref{eq:changeofbasis}) acts on $u$ by multiplication by $d$.
\end{rmk}

We can use this data to given a presentation for the moduli of $p$-divisible groups of height $\geq 2$.  We recall that the Witt ring functor $\Wt$ is represented by the ring $W = \mb Z[a_0,a_1,a_2,\ldots]$.
\begin{prop}
Displays in matrix form of height $h$ and dimension $(h-1)$ are represented by the ring
\[
A = \mb Z[(\beta_n)_{ij},det(\beta)^{-1}] \cong W^{\otimes h^2}[det(\beta)^{-1}]
\]
The indices range over $n \in \mb N$, $1 \leq i,j \leq h$.  The element $det(\beta)$ is the determinant of the matrix $((\beta_0)_{ij})$.

Isomorphisms between displays are represented by the ring
\[
\Gamma = A[(\phi_n)_{ij}, det(\phi)^{-1}] \cong A \otimes W^{\otimes h^2}[det(\phi)^{-1}]
\]
The indices range over $n \in \mb N$, $1 \leq i,j \leq h$, with the convention that $(\phi_0)_{ij}$ is zero if $1 \leq i \leq (h-1)$, $j=h$.  The element $det(\phi)$ is the determinant of the matrix $((\phi_0)_{ij})$.

The ideal $J = (p,(\beta_0)_{hh})$ of $A$ is invariant.  A display represented by $A \to R$ for $R$ a formal $\mb Z_p$-algebra is nilpotent if and only if it factors through the completion of $A$ at this ideal.
\end{prop}

\begin{proof}
The ring $W^{\otimes h^2}$ represents the functor
\[
R \mapsto \{h \times h\text{ matrices with entries in }\Wt(R)\},
\]
and so the ring $A$ represents $h \times h$ invertible matrices $(\beta_{ij}) = B$ with entries in the Witt ring.  Similarly, the ring $\Gamma$ represents pairs of a display in matrix form and an isomorphism to a second display in matrix form, according to the change-of-coordinates formula (\ref{eq:changeofbasis}).

The change-of-coordinates formula, mod $I_R$, takes $\beta_{hh}$ to a unit times itself, so the ideal $(p,(\beta_0)_{hh})$ is invariant.

Suppose $R$ is a $\mb Z/p^k$-algebra and $A \to R$ represents a matrix $B$, which we view as the matrix form of the display.  The display is nilpotent as in section~\ref{sec:defs} if and only if
\[
(\beta_0)_{hh}^{p^n} \cdots (\beta_0)_{hh}^{p} \cdot (\beta_0)_{hh}
\]
is zero in $R/(p)$ for some $n$.  This is equivalent to $(\beta_0)_{hh}$ being nilpotent in $R$, and so the display is then nilpotent over $R$ if and only if the map $A \to R$ factors through a continuous map $\comp{A} \to R$.  The corresponding statement for formal $\mb Z_p$-algebras follows.
\end{proof}

\begin{cor}
The pair $(A,\Gamma)$ forms a Hopf algebroid, and the completion $(\comp A,\comp \Gamma)$ at the invariant ideal $J$ has an associated stack isomorphic to the moduli of $p$-divisible groups of height $h$, dimension $1$, and formal height $\geq 2$.
\end{cor}

\begin{proof}
The existence of a Hopf algebroid structure on $(A,\Gamma)$ is a formal consequence of the fact that this pair represents a functor from rings to groupoids.

Let ${\cal M}_p(h)$ be the moduli functor of $p$-divisible groups of height $h$ and dimension $1$.  The universal nilpotent display on $\comp{A}$ gives rise to a natural transformation of functors $\Spf(\comp A) \to {\cal M}_p(h)$, and Theorem~\ref{thm:zinkcontravariant} implies that the $2$-categorical pullback $\Spf(\comp A) \times_{{\cal M}} \Spf(\comp A)$ is $\Spf(\comp \Gamma)$.

The resulting natural transformation of groupoid valued functors from the pair $(\Spf(\comp A),\Spf(\comp \Gamma))$ to ${\cal M}_p(h)$ is fully faithful by Theorem~\ref{thm:zinkcontravariant}; it remains to show that the map from the associated stack is essentially surjective.  Given a $p$-divisible group $\mb G$ on $\mb Z/{p^k}$-scheme $X$ of height $h$, dimension $1$, and formal height $\geq 2$, there exists an open cover of $X$ by affine coordinate charts $\Spec(R_i) \to X$ and factorizations $\Spec(R_i) \to \Spf(\comp A) \to {\cal M}_p(h)$.  It follows that $(\comp A, \comp \Gamma)$ gives a presentation of the moduli as desired.
\end{proof}

\begin{rmk}
The Hopf algebroid described is unlikely to be the best possible.  Zink's equivalence of categories shows that locally in the Zariski topology, a general $p$-divisible group can be described in this general matrix form; it is possible that there are more canonical matrix forms locally in the flat topology.

For example, if $h=2$ then a general matrix form
\[
\begin{bmatrix}
\alpha & \beta \\ \gamma & \delta
\end{bmatrix}
\]
(with $\delta$ nilpotent) can be canonically reduced to the form
\[
\begin{bmatrix}
0 & 1 \\ \gamma' & \delta'
\end{bmatrix},
\]
and by adjoining elements to $R$ to obtain a solution of $f^2 t = t \gamma'$ we obtain a faithfully flat extension in which the matrix can canonically be reduced to the form
\[
\begin{bmatrix}
0 & 1 \\ 1 & \delta''
\end{bmatrix}.
\]
The existence of canonical forms at higher heights, as well as more explicit determination of the associated Hopf algebroids, merits further study.
\end{rmk}

\section{Deformation theory}
\label{sec:deform}

In this section we briefly study the deformations of displays in matrix form.  We note that \cite{zink} has already fully interpreted the deformation theory of displays in terms of the deformations of the Hodge structure $Q/I_R P$.  The approach there is specifically in terms of fixing deformations of $F$ and $V^{-1}$ and classifying possible deformations of the ``Hodge structure'' $Q$; for calculational reasons we will instead fix the deformation of $Q$ and study possible deformations of the operators.

Let $\GL_h \subset \mb A^{h^2}_{\mb Z_p}$ be the group scheme of $h \times h$ invertible matrices.  There is a projection map $\Spec(A) \to \GL_h$ classifying the map that sends a display represented by a matrix $B \in \GL_h(\Wt(R))$ to the matrix $w_0(B)$.  Let $p\co \GL_h \to \mb P^{h-1}$ be the projection map sending a matrix $(\beta_{ij})$ to the point with homogeneous coordinates $[\beta_{1h}:\beta_{2h}:\ldots:\beta_{hh}]$.

\begin{thm}
Let $k$ be a field, and $\Spec(k) \to \Spec(A) \subset \Spec(W)^{h^2}$ be a point that defines a nilpotent display over a field $k$ with matrix form $B$.  Then the composite map $\Spec(A) \to \GL_h \overto \mb P^{h-1}$ identifies the set of isomorphism classes of lifts of this display to $k[\epsilon]/(\epsilon^2)$ with the tangent space of $\mb P^{h-1}_k$ at $\Spec(k)$.
\end{thm}

\begin{proof}
Suppose we are given the matrix form $B$ of a display over $k$, and write in block form
\[
B = 
\left[\begin{array}{c|c}
  \alpha & \beta \\
\hline
  \gamma & v\delta
\end{array}\right] \in \GL_h(\Wt(k)).
\]
(Nilpotence of the display forces the final entry to reduce to zero in $k$.)  Given any lift of the display on $k$ to a display on $k[\epsilon]/\epsilon^2$, Nakayama's lemma implies that any lift of the basis of the display gives a basis of the lift, whose matrix form is a lift of the matrix form over $k$.

Lifts of the matrix form $B$ to $k[\epsilon]/\epsilon^2$ are precisely of the form $B + s$ where $s$ is a matrix in $\Wt(\epsilon k)$, as any such element is automatically invertible.  Applying the change-of-basis formula (\ref{eq:changeofbasis}), we find that the lifts isomorphic to this one are of the form
\[
\left(I + 
\left[\begin{array}{c|c}
  fa & b \\
\hline
  p \cdot fc & fd
\end{array}\right]\right)
(B + s)
\left(I + \left[\begin{array}{c|c}
  a & vb \\
\hline
  c & d
\end{array}\right]\right)^{-1}
\]
for $a$, $b$, $c$, $d$ matrices in $\Wt(\epsilon k)$.  The ideal $\Wt(\epsilon k)$ is square-zero and annihilated by $f$, so this reduces to
\[
B + \left(s +
\left[\begin{array}{c|c}
  0 & b \\
\hline
  0 & 0
\end{array}\right]B
- B \left[\begin{array}{c|c}
  a & vb \\
\hline
  c & d
\end{array}\right]\right).
\]
The space of all isomorphism classes of lifts is therefore the quotient of $M_h(\Wt(\epsilon k))$ by the subspace generated by elements of the form
\[
B \left(-\left[\begin{array}{c|c}
  a & 0 \\
\hline
  c & d
\end{array}\right] -
\left[\begin{array}{c|c}
  0 & vb\\
\hline
  0 & 0
\end{array}\right] +
B^{-1} \left[\begin{array}{c|c}
  b \gamma & 0\\
\hline
  0 & 0
\end{array}\right]\right).
\]
As $B$ is invertible, this subspace consists of the $h \times h$ matrices in $\Wt(\epsilon k)$ whose final column is congruent to a multiple of the final column of $B$ (mod $I_R$).

However, this coincides with the kernel of the (surjective) map on relative tangent spaces $\Spec(A) \to \mb P^{h-1}$ over $\mb Z_p$ at $\Spec(k)$.
\end{proof}

\begin{cor}
\label{cor:niletale}
Suppose that we are given a display $(P,Q,F,V^{-1})$ of height $h$ and dimension $(h-1)$ over a formal $\mb Z_p$-algebra $R$ in matrix form $(\beta_{ij})$.  Then $R$ gives a universal deformation of the associated $1$-dimensional $p$-divisible group at all points if and only if the map $\Spec(R) \to \mb P^{h-1}_{\mb Z_p}$ given by 
\[
[w_0 (\beta_{1h}):w_0(\beta_{2h}):\cdots:w_0 (\beta_{hh})]
\]
is \'etale.
\end{cor}

\begin{proof}
The ring $R$ gives a universal deformation of the associated $p$-divisible group at a residue field $\Spec(k) = x$ if and only if the completed local ring $\comp{R}_x$ is mapped isomorphically to the universal deformation ring of the $p$-divisible group, which is a power series ring $\mb W(k)\llbracket u_1,\ldots,u_{h-1}\rrbracket$.  In particular, $R$ gives a universal deformation at $x$ if and only if:
\begin{itemize}
\item $R$ is smooth over $\mb Z_p$ at $x$, and
\item the relative tangent space of $R$ over $\mb Z_p$ at $x$, which is the set of lifts $\Spec(k[\epsilon]/\epsilon^2) \to \Spec(R)$, maps isomorphically to the set of lifts of the display to $k[\epsilon]/\epsilon^2$.
\end{itemize}

However, because $\mb P^{h-1}_{\mb Z_p}$ is smooth over $\Spf(\mb Z_p)$, the map $\Spf(R) \to \mb P^{h-1}_{\mb Z_p}$ is \'etale at $\Spec(k)$ if and only if:
\begin{itemize}
\item $R$ is smooth over $\mb Z_p$ at $x$, and
\item the relative tangent space of $R$ over $\mb Z_p$ at $x$ maps isomorphically to the tangent space of $\mb P^{h-1}_{k}$ at $x$.
\end{itemize}
By the previous theorem, these conditions coincide.
\end{proof}

\section{Associated spectra}
\label{sec:spectra}

We recall a statement of Lurie's theorem (as yet unpublished) from \cite{goerss}.  We write ${\cal M}_p(h)$ for the moduli of $p$-divisible groups of height $h$ and dimension $1$, ${\cal M}_{FG}$ for the moduli of $1$-dimensional formal groups, and ${\cal M}_p(h) \to {\cal M}_{FG}$ for the canonical map representing completion at the identity.

\begin{thm}[Lurie]
\label{thm:lurie}
Let ${\cal X}$ be an algebraic stack, formal over $\mb Z_p$,     equipped with a morphism
\[
{\cal X} \to {\cal M}_p(h)
\]
classifying a $p$-divisible group $\mb G$.  Suppose that at any point
$x \in {\cal X}$, the complete local ring of ${\cal X}$ at $x$ is
mapped isomorphically to the universal deformation ring of the $p$-divisible group at $x$.  Then the composite realization problem
\[
{\cal X} \to {\cal M}_p(n) \to {\cal M}_{FG}
\]
has a canonical solution; that is, there is a sheaf of $E_\infty$ even
weakly periodic $E$ on the \'etale site of ${\cal X}$ with $E_0$ locally isomorphic to the structure sheaf and the associated formal group isomorphic to the formal group $\mb G^{for}$.  The space of all solutions is connected and has a preferred basepoint.
\end{thm}

We may then combine this result with Corollary~\ref{cor:niletale} to produce $E_\infty$-ring spectra associated to schemes or stacks equipped with an appropriate cover by coordinate charts carrying displays.  Rather than stating in maximal generality, we have the following immediate consequence.

\begin{thm}
\label{thm:main}
Suppose $R$ is a formal $\mb Z_p$-algebra and $B$ is the matrix form of a nilpotent display over $R$, with associated $p$-divisible group $\mb G$.  If the associated map $\Spf(R) \to \mb P^{h-1}_{\mb Z_p}$ is \'etale, then there is an $E_\infty$ even-periodic $E = E(R,B)$ with $E_0 \cong R$, $E_2 \cong Q/I_R P$, and formal group isomorphic to the formal group $\mb G^{for}$.

Given matrix forms $B$, $B'$ of such displays over $R$ and $R'$ respectively, $g\co R \to R'$ a ring map, and $\phi$ a change-of-coordinates from $g^* B$ to $B'$ as in equation (\ref{eq:changeofbasis}), there exists a map of $E_\infty$ ring spectra $E(R,B) \to E(R',B')$ inducing $g$ on $\pi_0$ and lifting the associated map $\mb G^{for} \to (\mb G')^{for}$.  This construction is functorial in $B'$ as an object over $B$.
\end{thm}

\begin{proof}
The existence of $E(R,B)$ follows from Theorem~\ref{thm:lurie} and Corollary~\ref{cor:niletale}.  Given any such map $\Spf(R') \to \Spf(R)$, the maps $\Spf(R') \to \mb P^{h-1}$ and $\Spf(R) \to \mb P^{h-1}$ both being \'etale forces $\Spf(R')$ to be \'etale over $\Spf(R)$, and hence represents an element of the \'etale site.  Lurie's theorem then implies that the map lifts.

We recall from remark \ref{rmk:oneform} that in matrix form there is a nowhere vanishing $1$-form $u$ on the cotangent space $Q/I_R P$ of the formal group $\mb G^{for}$, implying that the tensor powers of the cotangent bundle are all free.  This implies the strictly even-periodic structure on $E$.
\end{proof}

Associated to a display not in matrix form, we would instead obtain a weakly even-periodic ring spectrum whose $2k$'th homotopy group is the $k$'th tensor power of the locally free module $Q/I_RP$ of invariant $1$-forms.

\begin{exam}
Let $h \geq 2$ and $R = \comp{(\mb Z[u_1,\cdots,u_{h-1}])}_{(p,u_1)}$.  Then we have the following display over $R$:
\begin{equation}
\label{eq:matrix}
\begin{bmatrix}
0 & 0 & 0 &  & 0 & 1 \\
1 & 0 & 0 & \cdots & 0 & [u_{h-1}]\\
0 & 1 & 0 &  & 0 & [u_{h-2}]\\
& \vdots &&&  & \vdots \\
0 & 0 & 0 &  & 0 & [u_2]\\
0 & 0 & 0 & \cdots & 1 & [u_1]
\end{bmatrix}
\end{equation}
Here $[x]$ denotes the Teichm\"uller lift of the element $x$.  The associated map $\Spf(R) \to \mb P^{h-1}$ is the map
\[
[1:u_{h-1}:\cdots:u_1],
\]
which is the completion of a coordinate chart of $\mb P^{h-1}$ at the ideal $(p,u_1)$ and is therefore \'etale.  Because this display is given in matrix form, there is a canonical non-vanishing invariant $1$-form $u$ and the resulting spectrum has homotopy groups
\[
\comp{(\mb Z[u_1,\ldots,u_{h-1}])}_{(p,u_1)} [u^{\pm 1}].
\]
Here $|u| = 2$.  This represents a ``partial'' thickening of a Lubin-Tate spectrum of height $h$ to a global object.
\end{exam}

\begin{exam}
Let $S = \comp{\Wt(\mb F_{p^h})[u_1,\cdots,u_{h-1}]}_{(p,u_1)}$, with the same display as given in the previous example.  Let $\zeta$ be a primitive $(p^h-1)$'st root of unity in $\mb F_{p^h}$ with Teichm\"uller lift $[\zeta]$.  Then the group $\mb F_{p^h}^\times = \langle \zeta \rangle$ acts on $S$ with generator $\zeta$ acting by
\[
\zeta \cdot (u_1,\ldots,u_{h-1}) = (\zeta^{1-p} u_1, \zeta^{1-p^2} u_2,\ldots,\zeta^{1-p^{h-1}} u_{h-1}),
\]
and the action of $\zeta$ lifts to an action on the display via the change-of-coordinates matrix
\[
\begin{bmatrix}
[\zeta^{p^{h-1}}] & 0 & & 0 & 0 \\
0 & [\zeta^{p^{h-2}}] & \cdots & 0 & 0\\
& \vdots && \vdots & \\
0 & 0 & \cdots & [\zeta^p] & 0 \\
0 & 0 & & 0 & [\zeta] \\
\end{bmatrix}.
\]
This acts on the invariant $1$-form $u$ by multiplication by $\zeta$.  This gives a well-defined action on the $p$-divisible group associated to the display.

Similarly, there is a Galois automorphism $\sigma$ of $S$ which acts by the Frobenius on $\Wt(\mb F_{p^h})$ and acts trivially on the generators $u_i$.  This automorphism preserves the display, and satisfies the relations $\sigma^h = id$, $\sigma \zeta = \zeta^p \sigma$.  Together these give an action of $G = (\mb F_{p^h}^\times
\rtimes {\rm Gal}_{\mb F_{p^h}/\mb F_p})$ on $\Spf(S)$ which lifts to an action on the associated $p$-divisible group.

Theorem~\ref{thm:main} implies that this lifts to an associated spectrum with an action of $G$.  (More generally, there is an associated sheaf of $E_\infty$-ring spectra on the quotient stack $[\Spf(S)//G]$.) The canonical invariant $1$-form $u$ is acted on by $\zeta$ by left multiplication, and acted on trivially by $\sigma$.  The homotopy fixed point object (which is the global section object of the quotient stack) has homotopy groups
\[
\left(\comp{\Wt(\mb F_{p^h})[u_1,\ldots,u_{h-1}]}_{(p,u_1)} [u^{\pm 1}]\right)^{G} \cong \comp{\mb Z[v_1,\ldots,v_{h-1},v_h^{\pm 1}]}_{(p,u_1)}.
\]
Here $v_i = u^{p^i - 1} u_i$ has degree $2p^i - 2$.  This has the homotopy type of a Johnson-Wilson spectrum completed at the height $\geq 2$ locus.
\end{exam}

\section{The period map}
\label{sec:period}

In section~\ref{sec:deform} we associated to each display in matrix form over $R$ a map $\Spec(R) \to \mb P^{h-1}$.  In this section we will briefly relate this, in a specific choice of coordinates, to the rigid analytic period map constructed by Gross and Hopkins \cite{grosshopkins}.  We first recall the construction of this period map.

Let $k$ be a perfect field of characteristic $p$ carrying a formal group law of finite height $h$, and $K = \Wt(k) \otimes \mb Q$.  Associated to this formal group law there is a universal deformation to a formal group law over the Lubin-Tate ring $R \cong \Wt(k)\pow{u_1,\ldots,u_{h-1}}$.  The formal group law gives rise to a Dieudonn\'e crystal on $R$ of rank $h$, and for a sufficiently large ring $R \subset S \subset K\pow{u_1,\ldots,u_{h-1}}$ the horizontal sections of this crystal on $\Spec(S)$ form a vector space $V$ over $K$ of rank $h$, containing a family of rank $(h-1)$-submodules determined by the Hodge structure.  The Gross-Hopkins period map $\Spec(S) \to \mb P(V)$ sends points of this rigid-analytic extension of Lubin-Tate space to the Hodge structure at that point.

Let $R = \Wt(k)\pow{u_1,\ldots,u_{h-1}}$, equipped with the ring homomorphism $\sigma\co R \to R$ which acts by the Frobenius on $\Wt(k)$ and sends $u_i$ to $u_i^p$; this is a lift of the Frobenius map on $R/p$, and provides a splitting $R \to \Wt(R)$ commuting with the Frobenius.  We write $J$ for the ideal $(u_1, \ldots, u_{h-1})$.  This ring $R$ carries the display of equation~\ref{eq:matrix}. This display is a universal deformation of a $p$-divisible group of height $h$ over the residue field $k$, and so the associated $p$-divisible group on $\Spf(R)$ is a universal deformation of the formal group on $\Spec(k)$.  The map $\Spec(R) \to \mb P^{h-1}$ of section~\ref{sec:deform} is the map $[1:u_{h-1}:\cdots:u_1]$.

To translate this into the (covariant) language of the Gross-Hopkins map, we first convert the display into the dual, covariant, display, which is a free $\Wt(R)$-module $P^t$ with dual basis $e^1, \ldots, e^h$ and Hodge structure $Q^t \subset P^t$ generated by $e^1, \ldots, e^{h-1}$.  (This Hodge structure is determined by the linear functional $\sum a_ie^i \mapsto a_h$.)  A straightforward calculation finds that the matrix of $F^t$ with respect to this basis is
\begin{equation}
\label{eq:frobenius}
\begin{bmatrix}
0 & p & 0 &  & 0 & 0 \\
0 & 0 & p & \cdots & 0 & 0\\
0 & 0 & 0 &  & 0 & 0\\
& \vdots &&&  & \vdots \\
0 & 0 & 0 &  & p & 0\\
0 & 0 & 0 &  & 0 & 1\\
p & p[u_{h-1}] & p[u_{h-2}] & \cdots & p[u_2] & [u_1]
\end{bmatrix}.
\end{equation}

As in \cite{zink}, there is a Dieudonn\'e crystal associated to this display.  The data of such a crystal produces: a module $M = R \otimes_{\Wt(R)} P^t$, a Hodge structure $Q^t/I_R P^t \subset M$, a $\sigma$-semilinear Frobenius map $F\co M \to M$, and a $\sigma$-antisemilinear map $V\co M \to M$ satisfying $FV = VF = p$.  Associated to this data there is a unique connection $\nabla \co M \to M \otimes \Omega_{R/\Wt(k)}$ for which $F$ and $V$ are horizontal.

Let $\Psi$ be the matrix of $F$ in this basis (the reduction mod $I_R$ of equation \ref{eq:frobenius}) and $\overline \Psi$ the image given by sending $u_i$ to $0$.  There exists a deformation of the basis $\{e^i\}$ to a basis of horizontal sections for this connection; i.e., there is a matrix $A \in \GL_h(K\pow{u_1,\ldots,u_{h-1}})$ whose columns are annihilated by $\nabla$ satisfying $A \equiv I$ mod $J$.  The expression in this basis for the linear functional cutting out the Hodge structure is given by the last row of $A$.

As $F$ is horizontal, it takes horizontal sections to horizontal sections, and hence applying the Frobenius to the columns of $A$ gives a linear combination of the combinations of $A$.  This implies that $\Psi A^\sigma = A B$ for some matrix $B$ with coefficients in $K$.  Reducing mod $J$ we find that $B = \overline \Psi$.  Thus, such a matrix must satisfy $A = \Psi A^\sigma \overline \Psi^{-1}$.  (Taking a limit of iterative substitutions recovers $A$ itself.)

As $A \equiv I$ mod $J$, $A^\sigma \equiv I$ mod $J^p$.  Therefore, we find $A \equiv \Psi \overline \Psi^{-1}$ mod $J^p$.  Applying this to equation~\ref{eq:frobenius}, we find
\[
A \equiv 
\begin{bmatrix}
1 & 0 &  & 0 & 0 \\
0 & 1 & \cdots & 0 & 0\\
& \vdots &&  & \vdots \\
0 & 0 &  & 1 & 0\\
u_{h-1} & u_{h-2} & \cdots & u_1 & 1
\end{bmatrix}\text{ mod }J^p.
\]
As a result, the Gross-Hopkins map classifying the Hodge structure is congruent to $[u_{h-1}:\cdots:u_1:1]$ mod $J^p$.  With an appropriate choice of coordinates we can then regard the map defined in section~\ref{sec:deform} as an approximation of the Gross-Hopkins map.

\nocite{*}
\bibliography{displays}

\end{document}